\documentclass[letterpaper,10pt,conference]{ieeeconf}  % Comment this line out
\IEEEoverridecommandlockouts                              % This command is only
\pdfoutput=1
\usepackage{amsmath}
\usepackage{amsthm,amsfonts}
\usepackage{amssymb}
\usepackage{mathtools}
\usepackage{xcolor}
\usepackage{bm}
\usepackage{float}
\usepackage{comment}
\usepackage{subcaption}    
\usepackage[font={small}]{caption}
\newcommand{\ec}{ {\bf\color{red}{EY:} }}
\newcommand{\nb}{ {\bf\color{green}{NB:} }}
\newtheorem{definition}{Definition}
\newtheorem{problem}{Problem}

\newtheorem{theorem}{Theorem}

\title{\LARGE \bf
Koopman Operators for Generalized Persistence of Excitation Conditions for Nonlinear Systems 
}
\author{Nibodh Boddupalli, Aqib Hasnain, Sai Pushpak Nandanoori, and Enoch Yeung% <-this % stops a space
\thanks{Please address correspondence to Nibodh Boddupalli at \tt\small nibodh@ucsb.edu.}% <-this % stops a space
\thanks{Nibodh Boddupalli, Aqib Hasnain, and Enoch Yeung  are with the Department of Mechanical Engineering and the Center for Control,Dynamical Systems, and Computation at the University of California Santa Barbara, Santa Barbara, CA 93106, USA.  Sai Pushpak Nandanoori  is with the Energy and Environment Directorate at the Pacific Northwest National Laboratory.
}}

\begin{document}

\maketitle
\thispagestyle{empty}
\pagestyle{empty}

%%%%%%%%%%%%%%%%%%%%%%%%%%%%%%%%%%%%%%%%%%%%%%%%%%%%%%%%%%%%%%%%%%%%%%%%%%%%%%%%
\begin{abstract}
It is hard to identify nonlinear biological models strictly from data, with results that are often sensitive to experimental conditions.  Automated experimental workflows and liquid handling enables unprecedented throughput, as well as the capacity to generate extremely large datasets.  We seek to develop generalized identifiability conditions for informing the design of automated experiments to discover predictive nonlinear biological models.   For linear systems, identifiability is characterized by persistence of excitation conditions.  For nonlinear systems, no such persistence of excitation conditions exist.   We use the input-Koopman operator method to model nonlinear systems and derive identifiability conditions for open-loop systems initialized from a single initial condition.  We show that nonlinear identifiability is intrinsically tied to the rank of a given dataset's power spectral density, transformed through the lifted Koopman observable space.  We illustrate these identifiability conditions with a simulated synthetic gene circuit model, the repressilator.  We illustrate how rank degeneracy in datasets results in overfitted nonlinear models of the repressilator, resulting in poor predictive accuracy.  Our findings provide novel experimental design criteria for discovery of globally predictive nonlinear models of biological phenomena. 

%Increasing scale and throughput brought forth by automation of experiments, coupled with unprecedented data-representations, warrants a design of experiments framework for efficient identifiability of the underlying complex systems. 

%While linear systems enjoy guaranteed identifiability governed by the notion of persistence of excitation, nonlinear systems have no such framework under system identification. In the case of an open-loop nonlinear system, designing an experiment for identifiability relies on the selection of conditions that initiate flows spanning all the invariant regions of its state-space representation - which becomes all the more challenging for multi-dimensional systems. This work investigates persistence of excitation conditions for nonlinear dynamical systems using Koopman operator theory. The Koopman operator presents a linear lifted representation of the evolution of a nonlinear system over a space of observable functions.  We use this framework to derive identifiability conditions, relating classical notions of persistence of excitation to the spectral properties of initial conditions, seen through the lens of Koopman observable functions.  We illustrate these concepts on a repressilator circuit, a classical biological system from the field of synthetic biology, to motivate the proper selection of initial conditions.  We show that when the rank properties of the flow's spectra are degenerate, the trained Koopman operator  does not produce robust predictions on test data.
\end{abstract}

\begin{comment}
\ec{I have a different perspective (and disagree with this sentence) - automation and increase in scale of throughput of experiments and the increase in data-representations of complex systems is what is driving this work}
\ec{good, keep maybe use open-loop nonlinear systems for brevity} 
\ec{not the right way to word this, revise and figure out by tonight}. 
\ec{why the repressilator? what would make this a compelling application problem? When would oyu DOE for a genetic oscillator? allude to this}.

\ec{good}

\ec{This sentence says system id is good, OK lead in better lead in is that PE conditions define criteria for identifiability of linear models, need generalization to nonlinear} 

\ec{sentence is somewhat redundant, ID is obviously good for control - maybe talk about how non-PE data can lead to local identifability => successful control in only a  limited portion of the phase space}.

\ec{ok}.  

\ec{don't write this unless you want to}

\ec{why the repressilator? what would make this a compelling application problem? When would oyu DOE for a genetic oscillator? allude to this}
\nb{This part has been addressed in the last section. However, this is an illustration I believe that there is no necessary reason that some other nonlinear system cannot be used for the same.}
\end{comment}

%%%%%%%%%%%%%%%%%%%%%%%%%%%%%%%%%%%%%%%%%%%%%%%%%%%%%%%%%%%%%%%%%%%%%%%%%%%%%%%%
\section{INTRODUCTION}

%Many physical systems are governed by nonlinear dynamics, some of these systems are extremely difficult to solve. 

Many physical systems exhibit phenomena with unknown governing dynamics.  These systems can be high dimensional and partially modeled or completely unstudied. Self-assembling complex systems, biological networks, Internet-of-Things infrastructure models, smart cities, and social networks are all examples of dynamically evolving systems that frequently are represented by data.   Identifying globally predictive models from data requires data collection strategies that capture all the modes of a system.  For example, in biological network modeling and discovery, designing an informative set of experimental conditions can produce global biological models that capture multiple modes of dynamics, including invariant subspaces and multiple equilibria.    Unfortunately, there are few identifiability metrics for quantifying the richness, or the {\it informativity}, of datasets of nonlinear systems.  Further, a dataset may only elicit linear modes of a nonlinear system, even though nonlinear modes may lay dormant. In these scenarios, the accuracy of a model discovery algorithm is often confounded with the informativity or richness of a dataset used to train the model.

Nonlinear systems lack generalized criteria for quantifying the information content of a given dataset. Even if a nonlinear model is deemed globally identifiable, this may be under the assumption of continuous noise-free sampling (perfect data) of all states.    In linear systems theory, the informativity of a dataset is characterized by its ability to persistently excite all the modes in the transfer function. These conditions, referred to as persistence of excitation conditions, prescribe rank requirements on either time-series or spectral signals.  First introduced in \cite{Astrom}, PE is ``not (yet) consistently defined" \cite{2017PE}.  There are two classical approaches to modeling persistence of excitation \cite{Astrom, 2017PE} .  Firstly, defining criteria or spectral properties of a control input such that it can elicit a response from the system, thereby uniquely specifying the frequency response model of the system (\cite{Ljung,Tangirala}.  Secondly, requiring that a control input that is  non-zero in all channels, at least once throughout the time course (\cite{Boyd,Sastry}).  Persistence of excitation conditions can also inform the appropriate construction of a universally exciting input signal, e.g. the construction of the Pseudo-Random Binary Sequence \cite{ONelles}, which persistently excites all linear systems.  It is thus desirable to establish a modeling framework to link classical results in linear system identification \cite{Ljung,Tangirala,Sastry} for the treatment of nonlinear systems.

Recently, an emerging set of operator-theoretic tools have gained traction, centered on discovering linear representations of nonlinear dynamical systems in a lifted space of coordinates \cite{Mezic2004,Mezic05,MezicAnnualReview}. Originally derived for Hamiltonian systems \cite{Koopman}, numerical \cite{2008APS} and theoretical \cite{Mezic2017spectrum} techniques for Koopman operator theory enable input \cite{Gramians,EnochControl,Kic,EnochMarkets,Korda2018data,korda2018linear} and spectral modeling of nonlinear systems \cite{Mezic05,kaiser2017data}, deep-learning based models of nonlinear phenomena \cite{deepDMD,LuschDMD,johnsonSIL}, and study of chaotic and uncountable spectra arising in complex nonlinear phenomena \cite{Korda2018data,williams2015data,Kic,Mezic2004,Mezic05}.  In this work, we use the Koopman operator method (Sections \ref{Secn_Koopman} and \ref{Koopman_control}) to lift nonlinear systems into a linear space and thereby derive generalized persistence of excitation conditions (Section \ref{secn_PE}), following a similar approach taken in linear systems theory.  We illustrate how generalized PE conditions inform the design of initial conditions for simulated experiments on the three node repressilator (Section \ref{secn_simulation}).

\section{KOOPMAN OPERATOR THEORY}\label{Secn_Koopman}
In 1931, Koopman showed the existence of a coordinate transformation $\psi \in \mathbb{R}^{n_L}$ and a corresponding unitary operator ${\mathcal K}:\mathbb{R}^{n_L}\rightarrow\mathbb{R}^{n_L}$ for any Hamiltonian (non-dissipative) dynamical systems. He showed that the operator ${\cal K}$ and observable $\psi$ could be used to represent the time evolution of the underlying Hamiltonian system as a linear time-invariant system. This idea has been developed and generalized to other classes of nonlinear systems in recent years (\cite{Mezic2004,Mezic05,Mezic2017spectrum}, etc.), in the effort to find global representations instead of local approximations.

A discrete-time nonlinear dynamical system with state $\bm x_t \in \mathbb{R}^n$ at time $t \in \mathbb{N}$ under $\bm f: \mathcal{M} \rightarrow \mathcal{M}$ can be represented as:
\begin{equation}\label{dyn_sys}
    \bm x_{t+1} = \bm f(\bm x_t)
\end{equation}
Then functions $\psi \triangleq \{\psi_i\}_{i=1}^{\infty} \in \mathcal{F}, i \in \mathbb{Z}_+$ are called "observables," and represent a mapping from the state-space into a lifted set of coordinates. For example, $\psi$ may be a scalar observable comprised of nonlinearities, weighted combinations of the state, such as $\psi_1(\bm x_t) = x_{1,t} \text{, }\psi_2(\bm x_t) = x_{3,t}^{0.5}, \psi_3(\bm x_t) = x_{2,t} x_{n,t}^2$. The true Koopman observables are often approximated in numerical Koopman learning algorithms like extended DMD (eDMD), deepDMD, and Hankel dynamic mode decomposition using functions which define a generic basis on a Hilbert function space, e.g. radial basis functions (RBFs), Hermite polynomials, or a combination of such basis functions in deepDMD \cite{deepDMD}.

For an analytical function $f(\bm x_t)$, we know from \cite{Gramians} that there exists a countably infinite or finite dimensional Koopman operator $\mathcal{K}$ that acts linearly on the observable $\psi : \mathcal{M} \rightarrow \mathbb{R}$ under function composition, satisfying the Koopman equation 
%Specifically, given two observable functions $\psi_1, \psi_2$  and  constants $\alpha, \beta$. This effectively propagates any $\psi(\bm x_t)$ forward in time as:
%\begin{equation*}
%    \mathcal{K} (\alpha \psi_1 + \beta \psi_2)= \alpha \psi_1 \circ f+ \beta \psi_2 \circ f = \alpha \mathcal{K} (\psi_1) + \beta \mathcal{K} (\psi_2)
%\end{equation*}
%Specifically, we have for $\alpha=1$ and a given observable function $\psi(\bm x)$ the famous Koopman equation,
\begin{equation}\label{lifted_sys}
    \mathcal{K} \psi(\bm x_t) = \psi \circ \bm f(\bm x) = \psi \big(\bm f(\bm x_t)\big) = \psi(\bm x_{t+1}).
\end{equation}
This equation states that the action of the Koopman operator on an observable function is equivalent to the action of the observable function on the vector field of the state. 

From the above transient behaviour, one can see that this itself is a dynamical system in the space of observables whose time-evolution is governed by an infinite-dimensional operator that preserves the essence of (\ref{dyn_sys}). The transformation from (\ref{dyn_sys}) to (\ref{lifted_sys}) is frequently referred to as "lifting" since the observable $\psi(\bm x)\in\mathbb{R}^{n_L}$ usually has dimension $n_L \geq n.$  This is not always the case, but \cite{johnsonSIL} shows that the observables define an expansion of the nonlinearities in the governing equations.

\subsection{Finite Dimensional Approximations}

\section{Koopman Input Output Theory}\label{Koopman_control}
Our study of the dynamical system (\ref{dyn_sys}), ultimately will require treating the initial condition as an input signal (via a Kronecker delta function) to the system.  For clarity, we introduce the notion of an input-Koopman operator here, as well as the notion of the Koopman transfer function of a nonlinear system.  Given a discrete-time nonlinear system with analytic vector field $\bm f(\bm x,\bm u)$ and control input $\bm u_t \in \mathbb{R}^m$, we write the dynamics of the system as:
\begin{equation*}
\begin{split}
    \bm x_{t+1} &= \bm f(\bm x_t, \bm u_t)\\
    \bm y_{t} &=\bm h(\bm x_t)
    %\bm y_t &= \bm h(\bm x_t)
\end{split}
\end{equation*}
\cite{Kic} showed that an observable $\psi: \mathcal{M}\times\mathbb{R}^m \rightarrow \mathbb{R}^{n_L}$ can lift the above system to $\mathcal{F}$ such that:
\begin{equation*}
        \mathcal{K}\psi(\bm x_t,\bm u_t) = \psi(\bm f(\bm x_t,\bm u_t), \bm u_{t+1}) = \psi(\bm x_{t+1},\bm u_{t+1})
\end{equation*}

For an exogenous memoryless input, \cite{EnochControl} demonstrated that the above can be modified by splitting $\psi(\bm x,\bm u)$ into components $\psi_x(\bm x)$ and $\psi_u(\bm x,\bm u)$, where $\psi_x(\bm x)$ is a stacked vector valued observable of all scalar valued observable functions from $\psi(\bm x,\bm u)$ that do not depend on $\bm u$, and $\psi_u(\bm x,\bm u)$ is a stacked vector valued observable of all remaining terms. This results in the decomposed representation
\begin{equation}\label{control_sys}
\begin{aligned}
    \bm \psi_x(\bm x_{t+1})& = \bm K_x \bm \psi_x(\bm x_t) + \bm K_u \bm \tilde{\bm \psi}_u(\bm z_t)\\
    \bm y &= W_h \bm \psi(\bm x_t,\bm u_t)
    \end{aligned}
\end{equation}
where $\tilde{\psi}_u({\bm z({\bm u_t})}) \equiv \psi_u(\bm x,\bm u)$ and $\bm z_t$ is a stacked vector of all multivariate terms of $\bm u_t$ and $\bm x_t$.  This is a representation of the nonlinear system dynamics that is linear in the {\it lifted} state observable $\psi_x(\bm x)$ and the {\it lifted} input-state mixture observable  $\psi_u(\bm x,\bm u) = \tilde{\psi_u} ({\bm z_t})$.  We thus can define the {\it Koopman discrete time transfer} function as 
\begin{equation}
    G_K(\text{z}) = {\bf W_h}(\text{z}I-{\bf K}_x)^{-1} {\bf K_u }
\end{equation}
where we have assumed that $\bm y_t = \psi_x({\bm x}_t).$  The zeros and the poles of the transfer function are defined in the usual manner, but with respect to the transformations on the state $\psi_x(\bm x_t)$  and input $\tilde{\psi}_u(\bm z(\bm u_t)))$. A continuous time analogue of the discrete-time Koopman operator is readily derived using the Koopman generator, rather than the Koopman operator.  While the discrete-time formulation is chosen in this work, a similar treatment of the sequel can be used to develop persistence of excitation conditions for continuous-time nonlinear dynamical systems.

\begin{comment}
    \ec{paragraph break, tie the DoE problem to discrete formulation just mentioned, if possible -->}
\end{comment}
 
%In this work, we primarily concern ourselves with design of experiments and selection of initial conditions to ensure identifiability of $G_K(\text{z})$. What are the conditions under which it is possible to estimate a input-output transfer function model of the system given either 1) a fixed choice of data, or 2) the opportunity to design the collection of data to ensure identifiability?  
\section{Problem Statement}\label{Secn_prob_state}
%As suggested by the exposition in the previous section, a primary application of Koopman operator theory is its use in data-driven applications to discover the governing equations or to learn a model to simulate dynamical system response.    The use of Koopman operator theory to represent nonlinear systems has enabled new operator-theoretic approaches for nonlinear control \cite{MilanMPC}.  The advent of data-driven algorithms for training Koopman operators \cite{deepDMD,LuschDMD,eDMD,HankelDMD} raises the question of how to design experiments or the collection of data to ensure identifiability of the Koopman operator. In the absence of control, a system's model identifiability is a Boolean property that is either satisfied or not \cite{Ljung}.  However, when considering the Doe via the initial condition of a dynamical system $\bm x_{t_0}$ or the design of a control perturbation $\bm u_t$, identifiability can be viewed as a function of the design.  

In classical system identification theory, the design of an input or initial condition to guarantee identifiability is framed in terms of PE. One of the merits of a Koopman operator representation is that it provides a representation of a nonlinear system in linear coordinates.  This allows us to rigorously formulate the question of how to design an initial condition that guarantees identifiability of the model. Specifically, we consider two problems:

\begin{problem}[Identifiability of a Nonlinear System with Fixed Initial Conditions]
Given a discrete time-invariant autonomous nonlinear system of the form
\begin{equation}
\begin{aligned}
\bm x_{t+1} &= \bm f(\bm x_t), \,\,\, \,  \bm x_{t_0} \in X_0   \\
\bm y_t & = \bm h(\bm x_t) \equiv \bm x_t
\end{aligned}
\end{equation}
determine if the model $\bm f(\bm x)$ can be identified from the continuously sampled data stream $\bm x(t)$ and the set of initial conditions $X_0$.
\end{problem}
This problem is difficult to solve, as it couples the problem of nonlinear function regression of $\bm f(\bm x)$ and $\bm h(\bm x)$ with the requirement of characterizing the richness of a set of fixed initial conditions.  This problem is equivalent to the nonlinear state-space realization problem, given a single initial condition or set of initial conditions $X_0$.  

The other variant of this problem is where the initial conditions are design parameters in an experiment and can be set by the user. This scenario is especially common in synthetic biology where experiments are conducted with varying initial conditions of concentrations, pH, etc. and hence is of interest from a design of experiment (DoE) standpoint.
\begin{comment}
\ec{More on this!  Why is this the scenario we care about? }
\end{comment}
\begin{problem}[Identifiability of a Nonlinear System with Designed Initial Conditions]
Given a continuous time-invariant autonomous nonlinear system of the form and the design choice of initial conditions $X_0$
\begin{equation}
\begin{aligned}
\bm x_{t+1} &= \bm f(\bm x_t), \,\,\, \,  \bm x_{t_0} \in X_0   \\
\bm y_t & = \bm h(\bm x_t) \equiv \bm x_t
\end{aligned}
\end{equation}
Find $X_0$ that guarantees the model $\bm f(\bm x)$ can be identified from the continuously sampled data stream $\bm x(t)$. 
\end{problem}
This variant of the problem presumably has more degrees of freedom, namely $\dim(\bm x_{t_0})|X_0)| = n|X_0|$ to be precise. However, the challenge is to relate identifiability of the model to the initial condition, which for an unknown $\bm f(\bm x)$ is inherently difficult.   We now use the input-Koopman framework derived in the previous sections to reformulate the problem with a linear Koopman representation.  The recasting of the problem will permit extension of classical identifiability notions such as the design of PE or SR input signals or initial conditions.

\section{Persistence of Excitation Conditions for Nonlinear Systems with Koopman Operators}\label{secn_PE}

Our contribution in this paper is to formulate the problem of identifiability with fixed and designed initial conditions using the method of Koopman.  Once we have formulated the problem in the Koopman operator theoretic framework, we derive computational certificates for PE for fixed or designed initial conditions. This leads to an algorithm for selection of the initial conditions of a dynamical system given $\bm x_{t_0}.$  In both scenarios, we treat the system's initial condition as an input. 
Given a nonlinear discrete-time dynamic system of the form 
\begin{equation}
\begin{aligned}
    \bm x_{t+1} &= \bm f(\bm x_t) \\
    \bm y_t & = \bm h(\bm x_t) \equiv x_t
    \end{aligned}
\end{equation}
where $\bm x \in \mathbb{R}^n$ and $t \in \mathbb{Z},$ the Koopman equation for the corresponding system defines the action of an operator $\bm K$ on a vector valued observable $\bm \psi \in \mathbb{R}^{n_L}$ acting on the state $\bm x_t \in \mathbb{R}^n$, namely 
\[ 
\bm \psi(\bm f(\bm x)) = \bm K\bm \psi(\bm x_t)
\]
Then, (\ref{control_sys}) can be represented with initial state $\bm \psi(\bm x_{t_0})$ as an input:
\begin{equation*}
\begin{split}
    \bm \psi(\bm x_{t+1}) &= \bm K\bm \psi(\bm x_t) + (\bm \psi(\bm x_{t_0}) - \bm K\bm \psi(\bm x_t))\delta_{t,t_0-1}
\end{split}
\end{equation*}

We thus model the initial condition as a Kronecker delta input to the dynamical system.  Accordingly, we will abuse notation slightly and define the input of the system as $(\bm \psi(\bm x_{t_0}) - \bm K \bm \psi(x_t))\delta_{t,t_0-1} \equiv \bm \varphi(\bm u_t)$, which yields the classic input-state Koopman
representation for a nonlinear system 
\begin{equation}\label{Lin_sys}
\begin{split}
    \bm \psi(\bm x_{t+1}) &= \bm K\bm \psi(\bm x_t) + \bm \varphi(\bm u_t).
\end{split}
\end{equation}
 Notice that the state vector $\psi_x(\bm x)\in\mathbb{R}^{n_L}$  is a vector observable function on which the Koopman operator acts as a linear operator to update the state. In essence, by defining an appropriate 'lifting' or 'observable' function, we can treat the problem as a classical linear identifiability problem, with additional caveats imposed by the presence of the nonlinearities in $\psi_x(\bm x)$ and $\psi_u(\bm x,\bm u)$. As there are multiple definitions in the literature for PE \cite{Ljung,Sastry}, we chose a definition that enables relating identifiability to the initial condition of a transfer function.  We pose an extension of the definition from 
\cite{Ljung} which is that of SR from \cite{Sastry}, but in the Koopman framework.

\begin{definition}[Persistence of Excitation] A quasi-stationary discrete time observable $\bm \varphi(\bm u_t) \in \mathbb{R}^{n_L}$ is said to be persistently exciting of order $N \in \mathbb{Z}_+$ if the covariance matrix $\overline{\bm R}_\varphi(N)$ is positive definite:
\begin{equation}\label{covariance}
    \overline{\bm R}_\varphi(N) := \begin{bmatrix}
   \bm R_\varphi(0) & \cdots &\bm R_\varphi(N-1)\\
    \vdots & \ddots & \vdots\\
   \bm R_\varphi(-(N-1)) & \cdots &\bm R_\varphi(0)\\
    \end{bmatrix}
\end{equation}
where $\bm R_\varphi(k), k \in \mathbb{Z}$ is the auto-covariance of $\bm \varphi(\bm u_t)$ formulated as:
\begin{equation}\label{auto-covariance}
   \bm R_\varphi(k) \coloneqq \mathbb{E} \big[ \bm \varphi(\bm u_t) \bm \varphi(\bm u_{t+k})^T \big],
\end{equation}
\end{definition}
\noindent and ${ \mathbb{E}}$ denotes the expectation operator.  The above definition has been referred to as {\it sufficient richness} in \cite{Sastry} as well. Positive semi-definiteness of $\overline{\bm R}_\varphi(N)$ can easily be proved. By substituting the $R.H.S.$ of (\ref{auto-covariance}) in (\ref{covariance}) we obtain:
\begin{multline*}
    \overline{\bm R}_\varphi(N)=
    { \mathbb{E}}\bigg\{ \begin{bmatrix}
    \bm \varphi(\bm u_{t+1}) \\ \vdots \\ \bm \varphi(\bm u_{t+N}) \end{bmatrix}
    \begin{bmatrix}
    \bm \varphi(\bm u_{t+1})^T,\cdots,\bm \varphi(\bm u_{t+N})^T \end{bmatrix}\bigg\}
\end{multline*}
Defining $\bm v \triangleq \big[\bm \varphi(\bm u_{t+1})^T,\cdots,\bm \varphi(\bm u_{t+N})^T\big]^T$, the above becomes:
\begin{equation}
    \begin{split}
        \overline{\bm R}_\varphi(N) &= { \mathbb{E}} \big[ \bm v \bm v^T \big].\\
    \end{split}
    \end{equation}
Let $\bm q = [\bm q_1^T,\dots,\bm q_k^T,\dots, \bm q_N^T]^T, \bm q_k \in \mathbb{R}^{n_L}  \setminus \{0\},$ we have
\begin{equation}\label{SR}
    \begin{split}
        \bm q^T \overline{\bm R}_\varphi(N) \bm q &= { \mathbb{E}} \big[ \|\bm q^T \bm v\|_2^2\big] \\ & = \sum_{l,m = 1}^N \bm q_l^T\bm R_\varphi(m-l) \bm q_m \\ &\geq 0\\
    \end{split}
\end{equation}

Hence, $\overline{\bm R}_\varphi(N)$ is positive semi-definite. From the Herglotz Theorem, any function $R_g(k)$ defined on integers is positive semi-definite if and only if it has a Bochner representation (can be represented as the inverse Fourier transform) of a unique Spectral Measure $S_g(\nu)$ on a circle:
\begin{align*}
    R_g(k) &= \int_{-\pi}^{\pi}e^{ik\nu} S_g(d\nu)
           = \int_{-\pi}^{\pi}e^{ik\nu} S_g(\nu) d\nu
\end{align*}

If $R_g(k)$ is a scalar, $S_g(\nu)$ is its power spectrum as pointed out in \cite{Sastry} and \cite{Boyd}. Alternate definitions have been mentioned in literature but the following definition and result elucidates how PE of an initial condition for a dynamical system can be related to the Koopman transfer function. 
\begin{definition}
Given system (\ref{dyn_sys}), we say an initial condition $x_0$, treated as a  Kronecker delta signal $\delta(\bm x_0)$ is persistently exciting of Koopman-order $n_L$ if it is persistently exciting for a state-inclusive Koopman operator ${\cal K} \in \mathbb{R}^{n_L}$ of order $n_L$ satisfying 
\begin{multline}\label{eq:KoopmanStateInclusive}
        \psi(\bm x_{t+1}) =\begin{bmatrix}
 \bm x_{t+1} \\ \varphi(\bm x_{t+1}) \end{bmatrix}= \begin{bmatrix} K_{xx} & K_{x\varphi} \\K_{\varphi x} &K_{\varphi\varphi}  
    \end{bmatrix}\begin{bmatrix}
 \bm x_{t} \\ \varphi(\bm x_{t}) \end{bmatrix} +\\
\delta_{t,-1} (\bm \psi(\bm x_0) - K\bm \psi(x_t)).
\end{multline}
\end{definition}
\begin{theorem}\label{prop}
The initial condition $x_0$ is persistently exciting for the nonlinear dynamical system (\ref{dyn_sys}) if and only if the Fourier transform of the auto-covariance matrix $\bm R_\varphi(k)$
\begin{equation*}\label{Wiener2}
    \bm S_\varphi(\omega) = \sum_{k=-\infty}^{\infty}\bm R_\varphi(k)e^{-ik\omega}
\end{equation*}
has $n_L$ distinct frequencies $\omega_1,...,\omega_{n_L}$ where $\bm S_\varphi(\omega)$ does not vanish, i.e. $n_L$ positive spectral lines.
\end{theorem}
\begin{proof}
Let the dimension of the lifted space and corresponding Koopman vector valued observable $\psi(\bm x)$  be denoted as  $n_L$.  We write the initial condition as an input to the nonlinear dynamical system  (\ref{dyn_sys}), of the form $\delta_t(\bm x_0)$.  We suppose that $\psi(\bm x)$ is state inclusive \cite{johnsonSIL}.  This implies that:
\begin{multline*}
x_{t+1} = f(\bm x_t) = K_{xx} x_t + K_{x\varphi}\varphi(x_t) \\
 \hspace{3mm}  + \delta_{t,-1} (\bm x_0 - K_{xx}\bm x_t - K_{x\varphi}\varphi(x_t))
\end{multline*}
i.e. the PE of an input signal $\psi(\bm x_0)$ for the system (\ref{eq:KoopmanStateInclusive}) implies the PE of the system (\ref{dyn_sys}), since the flow and vector field of system  (\ref{eq:KoopmanStateInclusive}) is a projection of the flow and vector field of the system (\ref{dyn_sys}), respectively. It suffices to demonstrate the equivalence of the PE of $\psi(\bm x_0)$, up to order $n_L$, to the linear independence of spectral lines for the power spectral density $\bm S_\varphi(\omega).$ 

We know that the spectral measure $\bm  S_\varphi(\omega)$ of $\bm R_\varphi(k)$ would be a positive semi-definite, symmetric matrix of bounded measures, symmetric over all $\omega$ since the observable elements $\varphi(\bm x)$ are all real valued. We have by the definition of the spectral power measure that
\begin{equation}\label{Bochner}
   \bm R_\varphi(k) = \int_{-\pi}^{\pi}e^{ik\omega} \bm S_\varphi(\omega)d\omega.
\end{equation}

Since the spectral measure is bounded almost everywhere, the monotone convergence theorem allows us to write from (\ref{SR})
\begin{equation*}
    \begin{split}
        \sum_{l,m = 1}^N \bm q_l^T \big[ \int_{-\pi}^{\pi}e^{i(m-l)\omega} \bm S_\varphi(\omega)d\omega \big] \bm q_m &\geq 0.\\
        \end{split}
        \end{equation*}
After swapping the finite sum with limiting sums
        \begin{equation*}
            \begin{split}
        \int_{-\pi}^{\pi} \bigg( \sum_{l,m = 1}^N \bm q_l^T e^{-il\omega} \bm S_\varphi(\omega) e^{im\omega} \bm q_m \bigg)d\omega &\geq 0, \\
        \end{split}
        \end{equation*}
and distributing sums across the integrand, we obtain
        \begin{equation*}
            \begin{split}
        \int_{-\pi}^{\pi} \bigg( \big( \sum_{l = 1}^N e^{-il\omega} \bm q_l^T \big) \bm S_\varphi(\omega) \big(\sum_{m = 1}^N e^{im\omega} \bm q_m \big) \bigg)d\omega &\geq0.
    \end{split}
\end{equation*}
Defining filters $\bm Q(e^{i\omega}) \triangleq \sum_{k = 1}^N e^{ik\omega} \bm q_k$, the above becomes:
\begin{equation}\label{possitive_defn}
    \int_{-\pi}^{\pi} \bigg( \bm Q^*(e^{i\omega}) \bm S_\varphi(\omega) \bm Q(e^{i\omega}) \bigg)d\omega \geq 0.
\end{equation}
Thus, positive definiteness of $\bm S_\varphi(\omega)$ only holds if and only if $\bm R_\varphi(k)$ is positive definite. But positive definiteness of $\bm S_\varphi(\omega)$ holds if and only if there are at least $n_L$ frequencies in which the integrand does not vanish, namely $n_L$ frequencies where \[
\sum_{k=1}^{N} e^{ik\omega}\bm q_k
\] is in the orthogonal complement of the null space of $\bm S_\varphi(\omega).$  This proves the result. 
%Since the filters $\bm Q(e^{i\omega})$ have at most $N-1$ distinct zeros, the integrand is always positive definite if $\bm S_\varphi(\omega)$ is non-singular at at least $N$ distinct frequencies $\omega$. By the Wiener-Khinchin Theorem, this matrix of power spectrum is the Fourier transform of the auto-covariance matrix \cite{Tangirala}:
%\begin{equation*}\label{Wiener}
%    \bm S_\varphi(\omega) = \sum_{k=-\infty}^{\infty}\bm R_\varphi(k)e^{-ik\omega}
%\end{equation*}
%The function $\bm S_\varphi(\omega)$ is thus is full rank for at least $N$ distinct frequencies \textcolor{red}{if $\Hat{\bm \varphi}(\bm u_\omega)$ has linearly independent amplitudes at those frequencies.}
\end{proof}
From our simulation studies, we found that most initial conditions are PE up to order $n_L$, for their respective Koopman operator and dynamical system.   If we consider the design problem of selecting an initial condition $\bm x_0$ such that $\psi(\bm x_0)$ is PE, or a set of $\Psi(\bm x_0)$ are PE of Koopman-order $n_L$, we can express the problem in terms of the positive definiteness of $\bm R_\varphi(N),$  adjusting the signal $\psi(\bm x_0)$ or more generally the timing of the input to ensure the positive definiteness of the auto-covariance matrix.  Alternatively, when working with a collection of initial condition signals $\psi(\bm x_0) \in \Psi(\bm x_0)$ it is straightforward to visualize the power spectrum using the transformed signal $\delta_t(\psi(\bm x_0)\in \Psi(\bm x_0)).$  Initial conditions can be selected or drawn randomly from the phase space until a suitable collection of initial conditions and $n_L$ spectral lines are identified. 
%Since the lifted space is $n_L$-dimensional, the input is needed to be PE of order $n_L$. Then, the filter $\bm Q(e^{i\omega})$ can be replaced with the transfer function of (\ref{Lin_sys}) and it can be seen that (\ref{possitive_defn}) is a mathematical formulation for the intuition that $\bm \varphi(\bm u_t)$ should contain frequencies besides those that are zeros of the transfer function. Since the input is the set of initial conditions, this is accomplished by initial conditions whose power spectrum in the lifted space is comprised of linearly-independent amplitudes at at-least $n_L$ different frequencies.

%This can easily be accomplished by a initializing the state $\bm x_{t_0}$ such that the observable of it is a Heaviside function of linearly independent amplitudes in the $n_L$-dimensions.

%\addtolength{\textheight}{-3cm}   % This command serves to balance the column lengths
                                  % on the last page of the document manually. It shortens
                                  % the textheight of the last page by a suitable amount.
                                  % This command does not take effect until the next page
                                  % so it should come on the page before the last. Make
                                  % sure that you do not shorten the textheight too much
\begin{figure}
    \centering
    \begin{subfigure}{\columnwidth} \centering
    \includegraphics[width=.95\columnwidth]{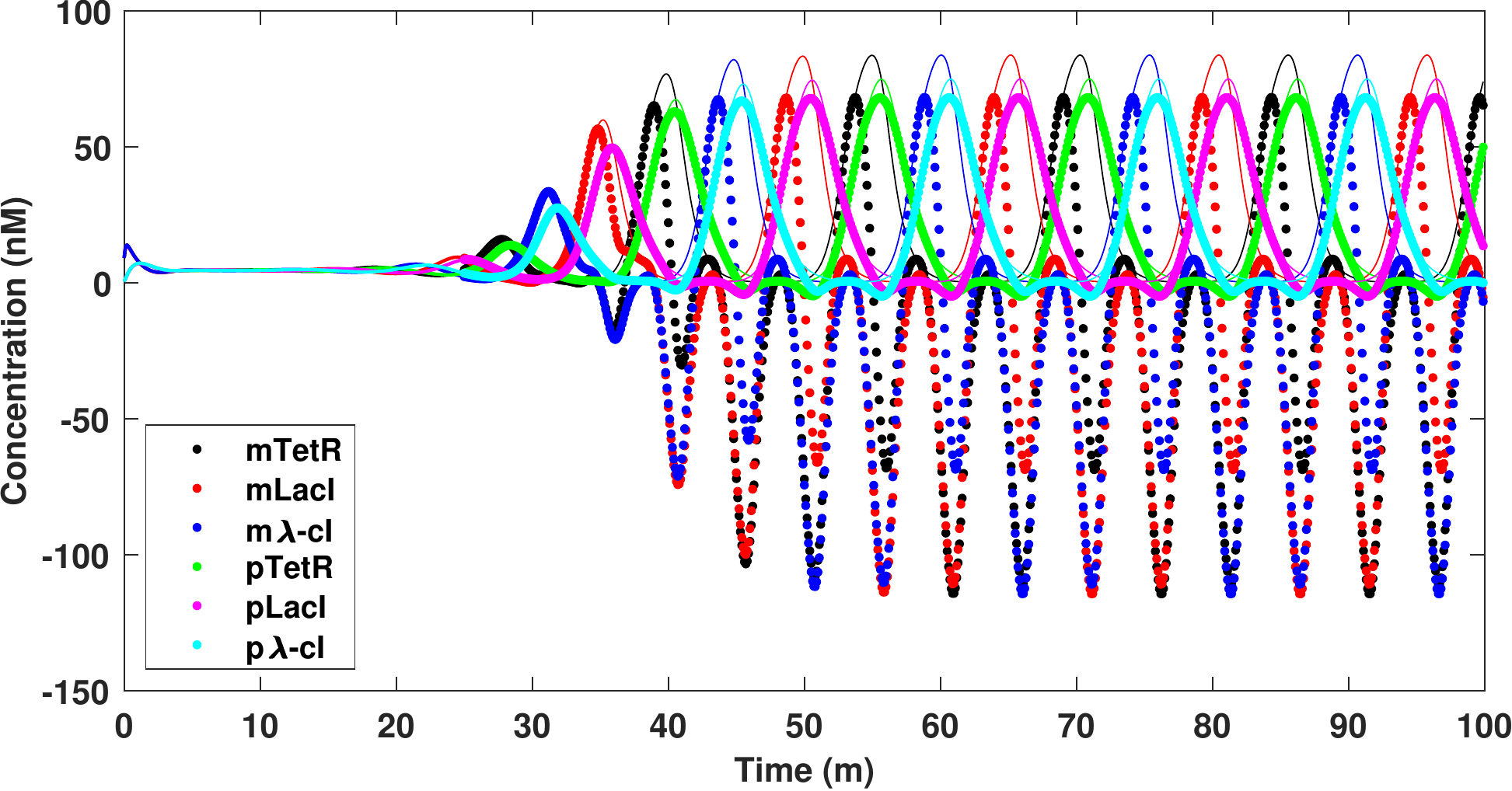}
    \caption{} \label{subfig_1a} \end{subfigure}
    \begin{subfigure}{0.49\columnwidth} \centering
    \includegraphics[width=.95\columnwidth]{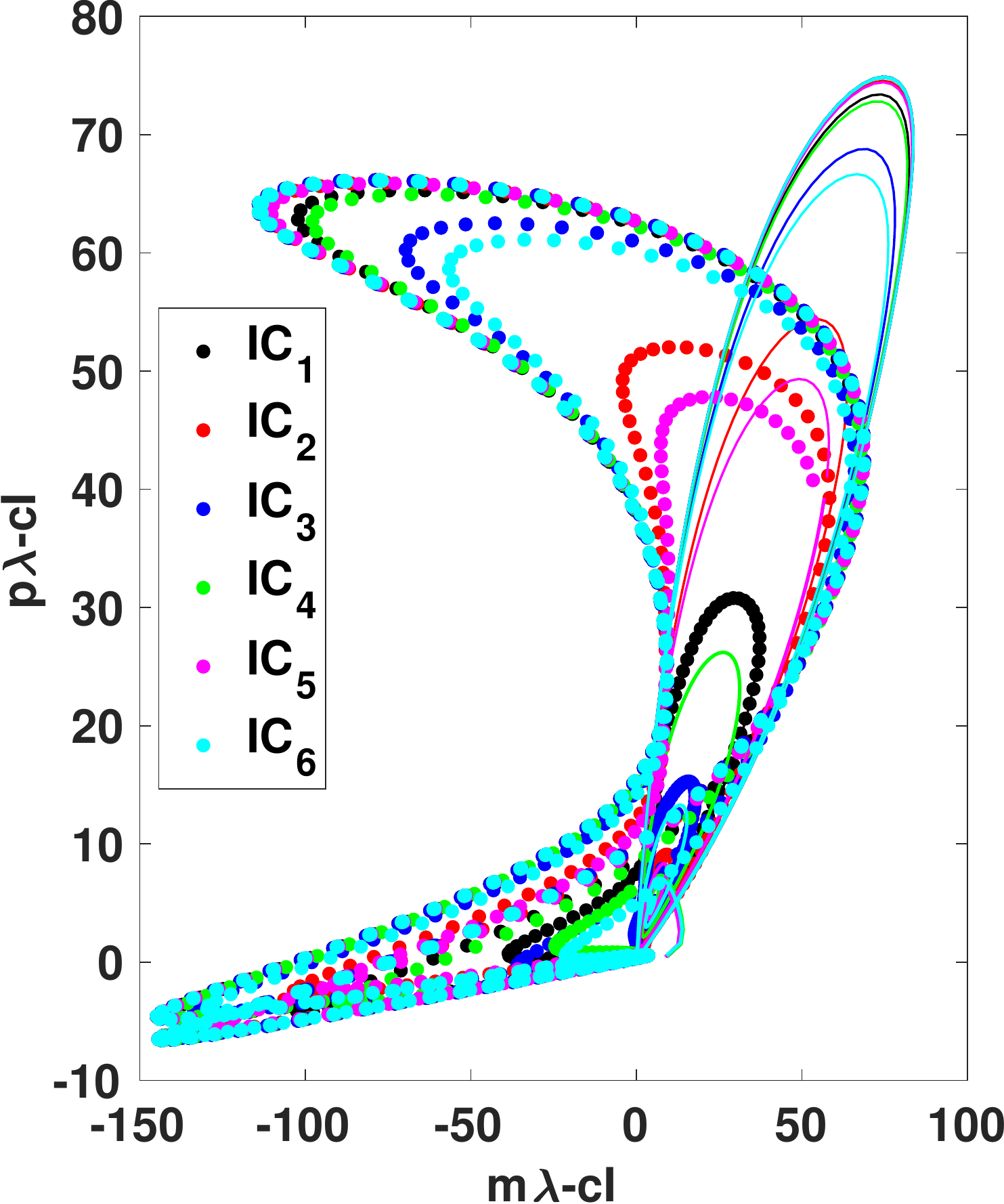}
    \caption{} \label{subfig_1b} \end{subfigure}
    \begin{subfigure}{0.49\columnwidth} \centering
    \includegraphics[width=.95\columnwidth]{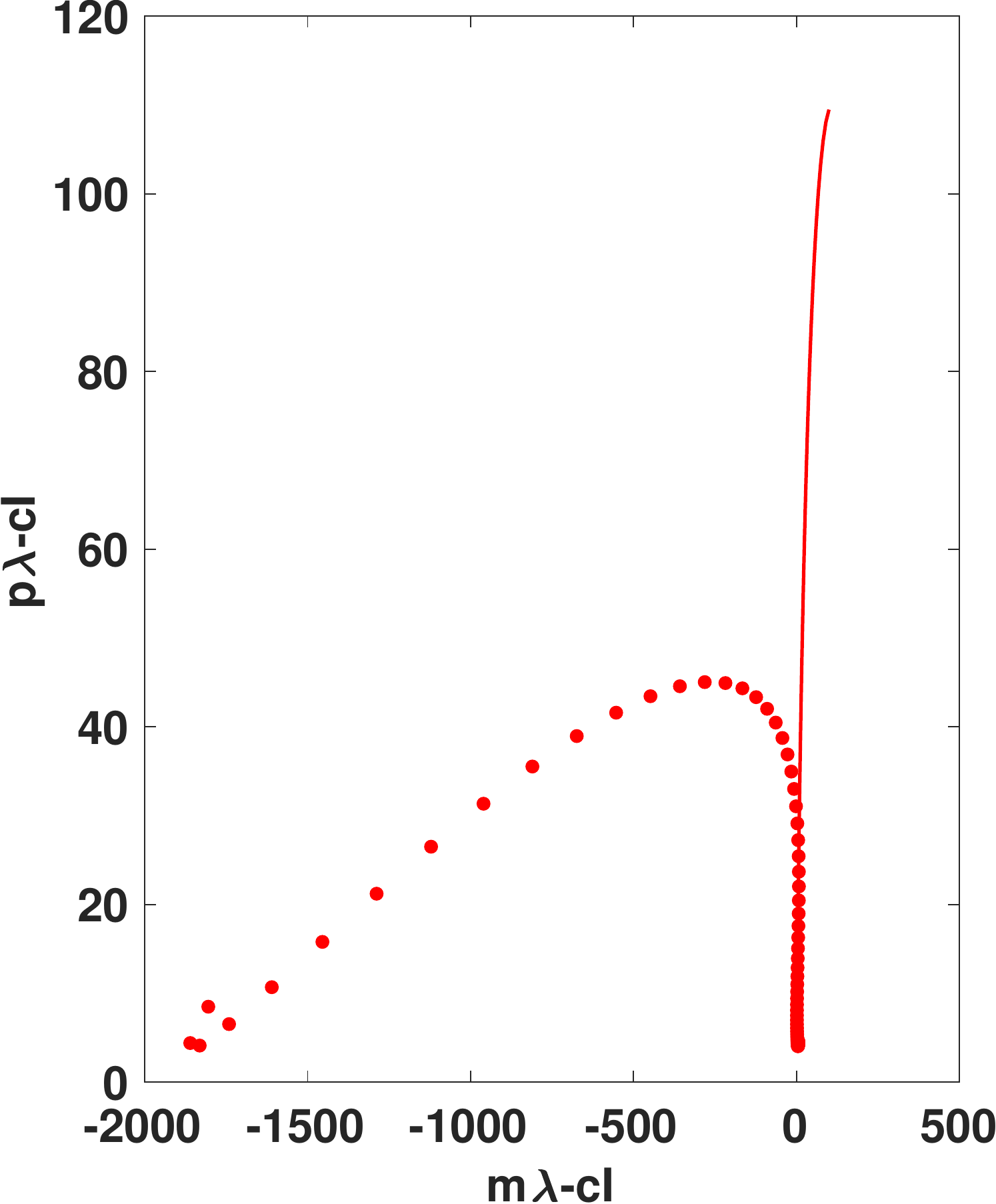}
    \caption{} \label{subfig_1c} \end{subfigure}
    \begin{subfigure}{0.49\columnwidth} \centering
    \includegraphics[width=.95\columnwidth]{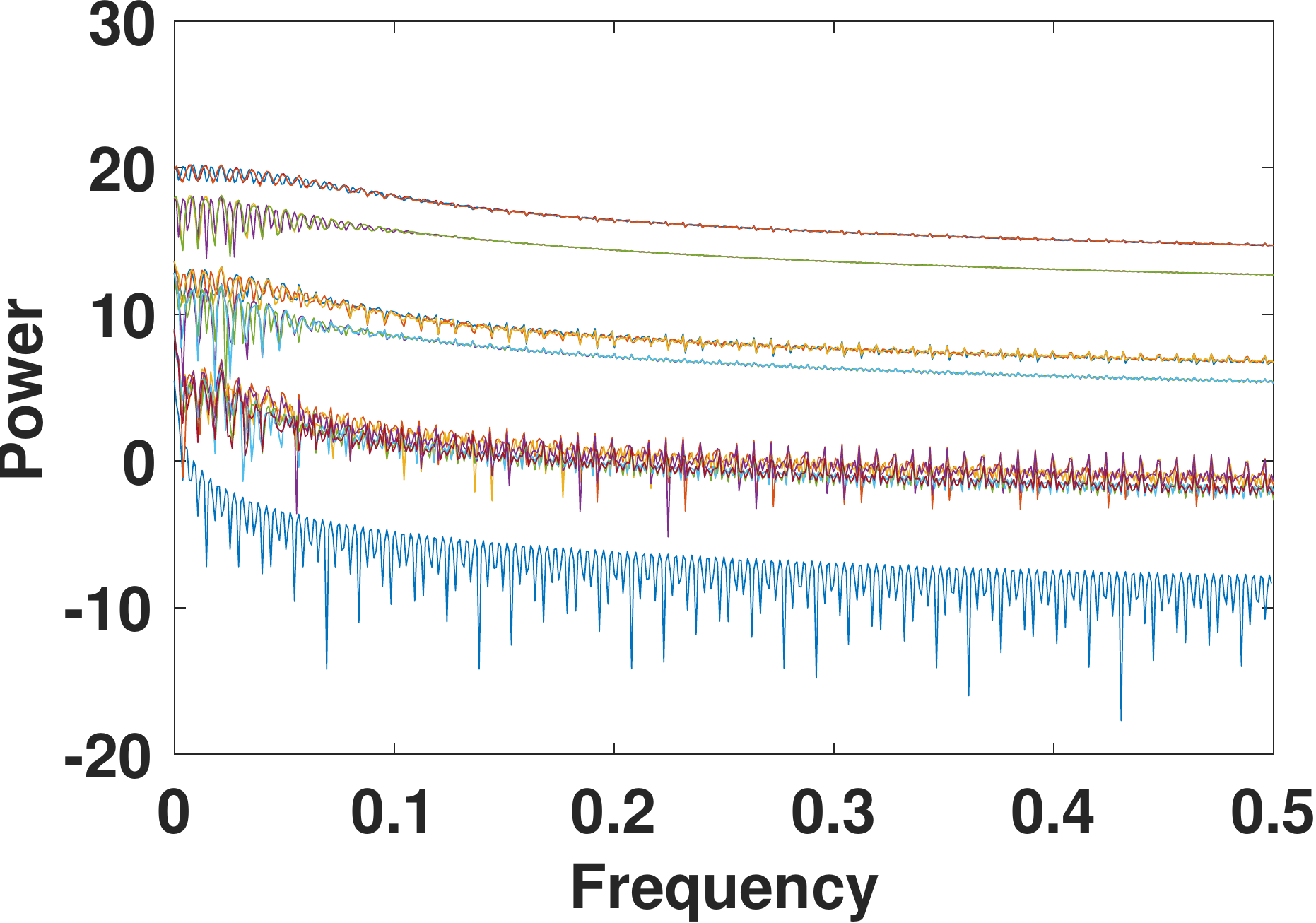}
    \caption{} \label{subfig_1e} \end{subfigure}
    \begin{subfigure}{0.49\columnwidth} \centering
    \includegraphics[width=.95\columnwidth]{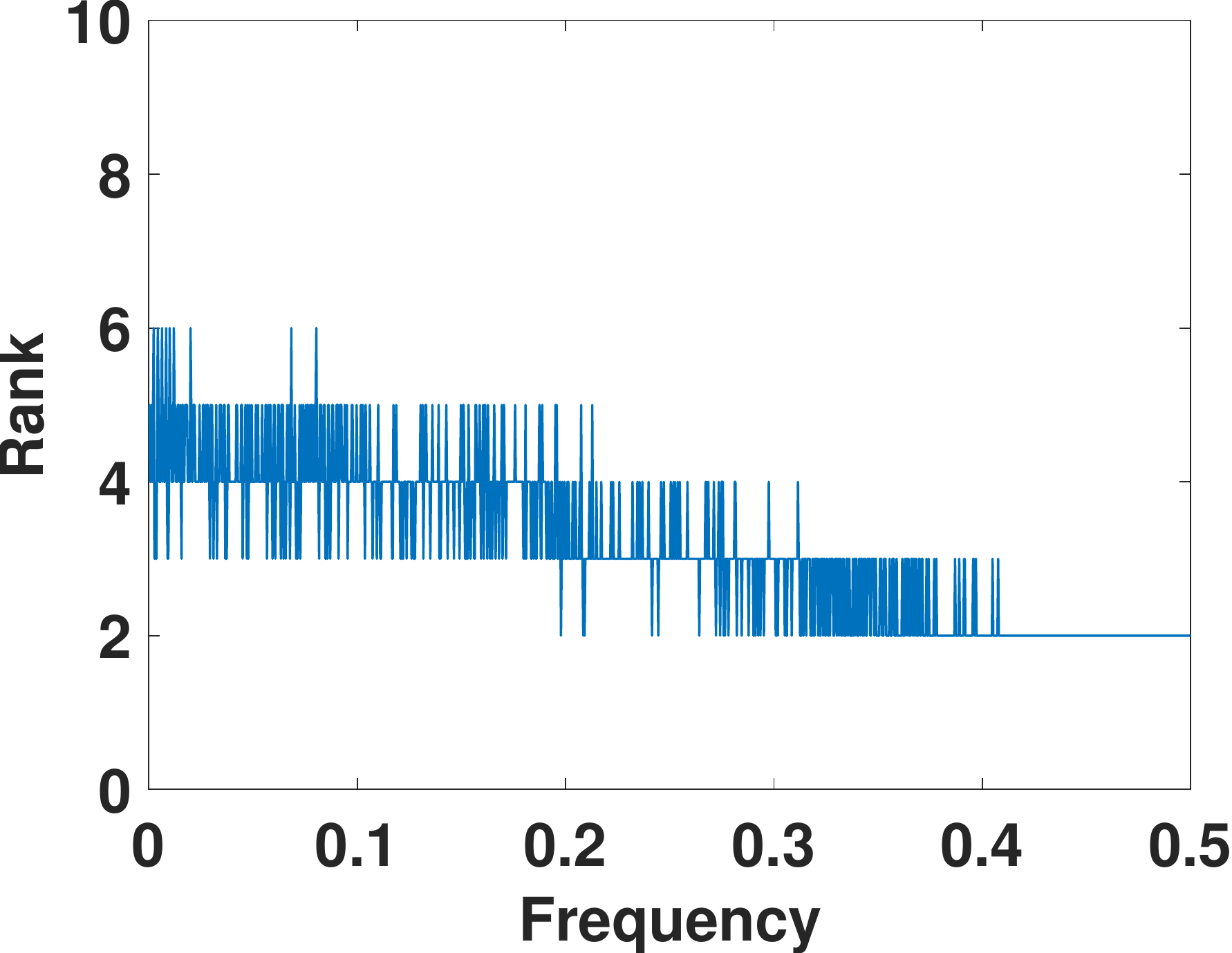}
    \caption{} \label{subfig_1f} \end{subfigure}
    \caption{The repressilator trained from inside a unit ball centered at $0$ and predicted from another basin of attraction using the Koopman operator. (a) States oscillating with time as simulated (solid lines). Koopman operator trained up to $t = 25$ for prediction (dotted lines) then onward. (b) Simulated (solid) trajectories growing into limit cycles and prediction (dotted) of the limit cycles. (c) Simulated (solid) and predicted (dotted) trajectory form another basin of attraction. (d) Periodogram of trajectories. (e) Rank of Power Spectrum of trajectories up-to $t = 25$ used to train the Koopman Operator}
    \label{fig:RepressilatorPlot_Inside}
\end{figure}                                  
\section{Persistence of Excitation of Different Initial Conditions for a Repressilator Genetic Circuit Model}\label{secn_simulation}

\subsection{The Repressilator}

The repressilator is a classical genetic circuit used in synthetic biology to implement circadian rhythms or synthetic oscillations.  The architecture is that of a 3 node Goodwin oscillator, with three genes that produce proteins or mRNA that serve to repress the downstream or target gene's function.  Each gene represses its downstream target, with the final gene repressing the original gene to form a cycle of negative feedback. When the gain of the individual genes are balanced with respect to each other \cite{Sontag}, the genetic circuit admits a limit cycle in the phase portrait and a single basin of attraction surrounding the origin.
\begin{figure}
    \centering
    \begin{subfigure}{\columnwidth} \centering
    \includegraphics[width=.95\columnwidth]{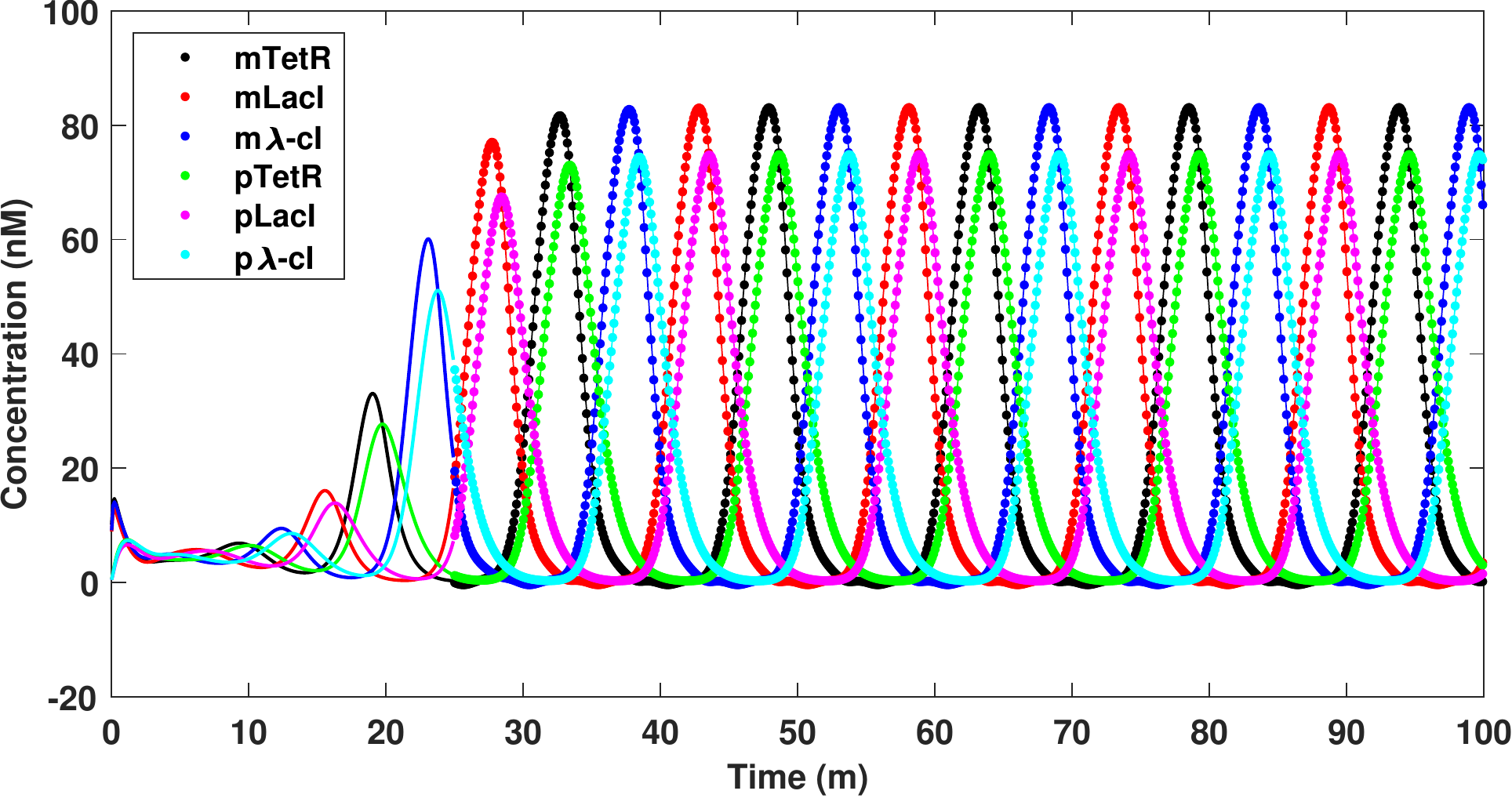}
    \caption{} \label{subfig_2a} \end{subfigure}
    \begin{subfigure}{0.49\columnwidth} \centering
    \includegraphics[width=.95\columnwidth]{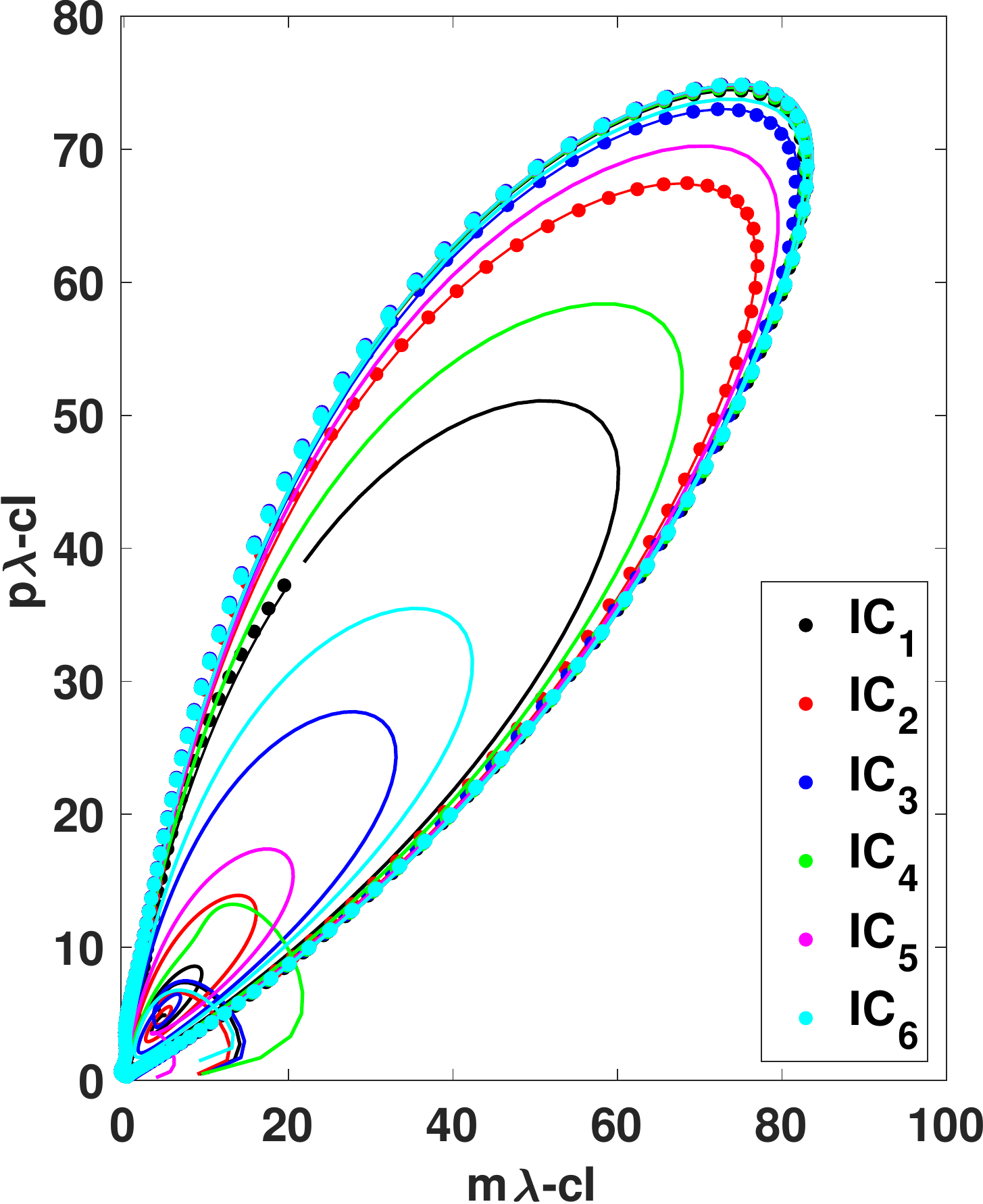}
    \caption{} \label{subfig_2b} \end{subfigure}
    \begin{subfigure}{0.49\columnwidth} \centering
    \includegraphics[width=.95\columnwidth]{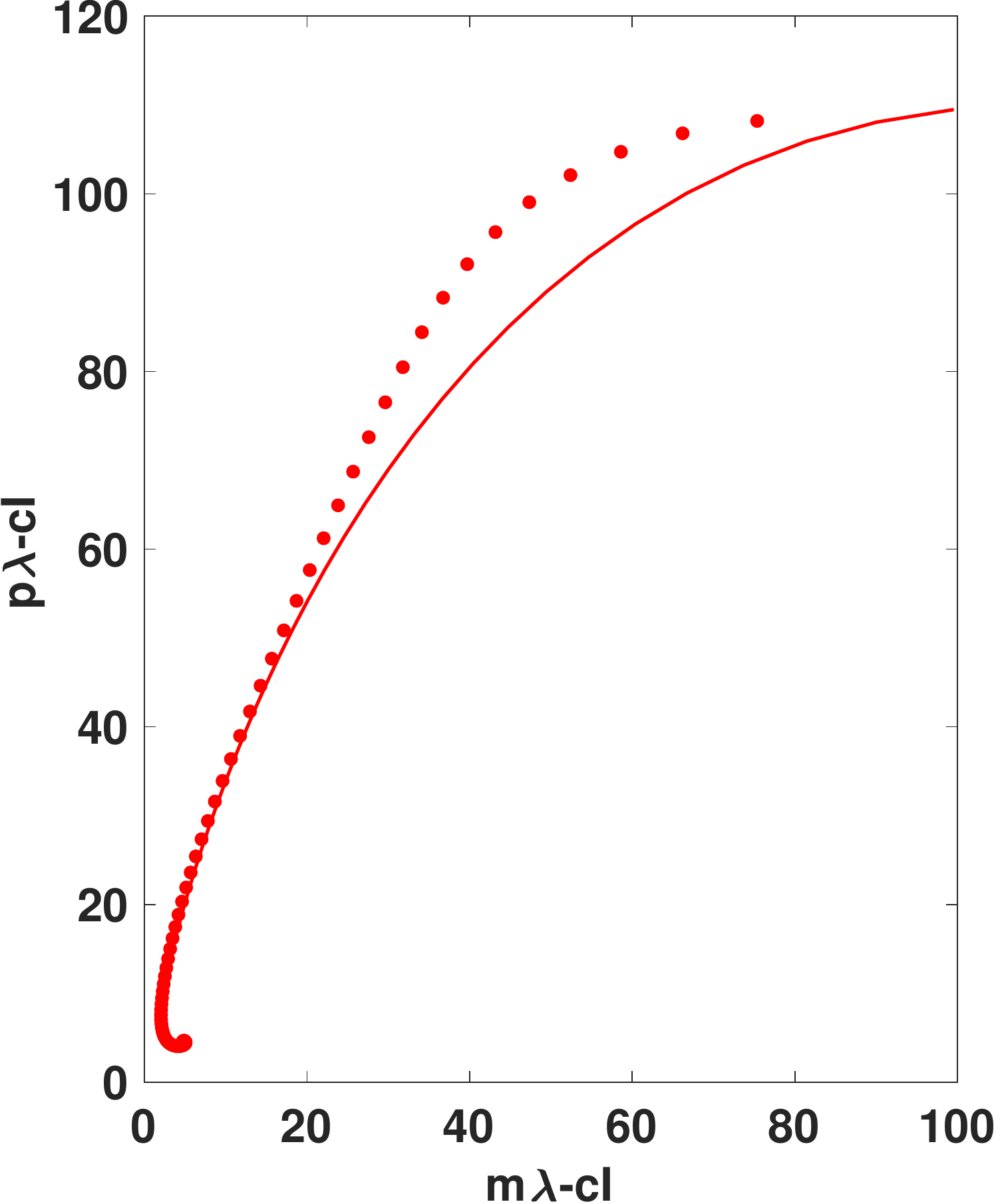}
    \caption{} \label{subfig_2c} \end{subfigure}
    \begin{subfigure}{0.49\columnwidth} \centering
    \includegraphics[width=.95\columnwidth]{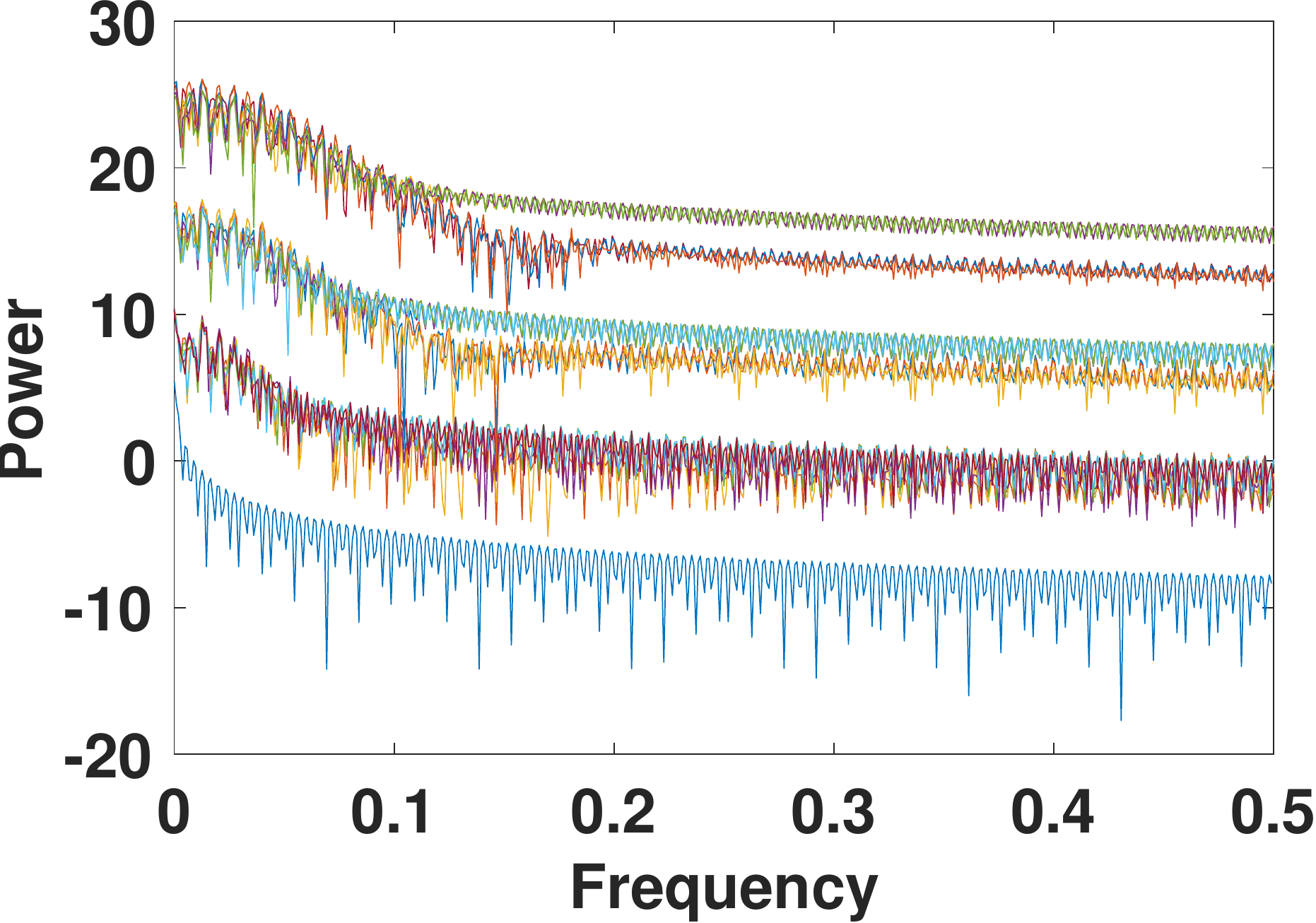}
    \caption{} \label{subfig_2e} \end{subfigure}
    \begin{subfigure}{0.49\columnwidth} \centering
    \includegraphics[width=.95\columnwidth]{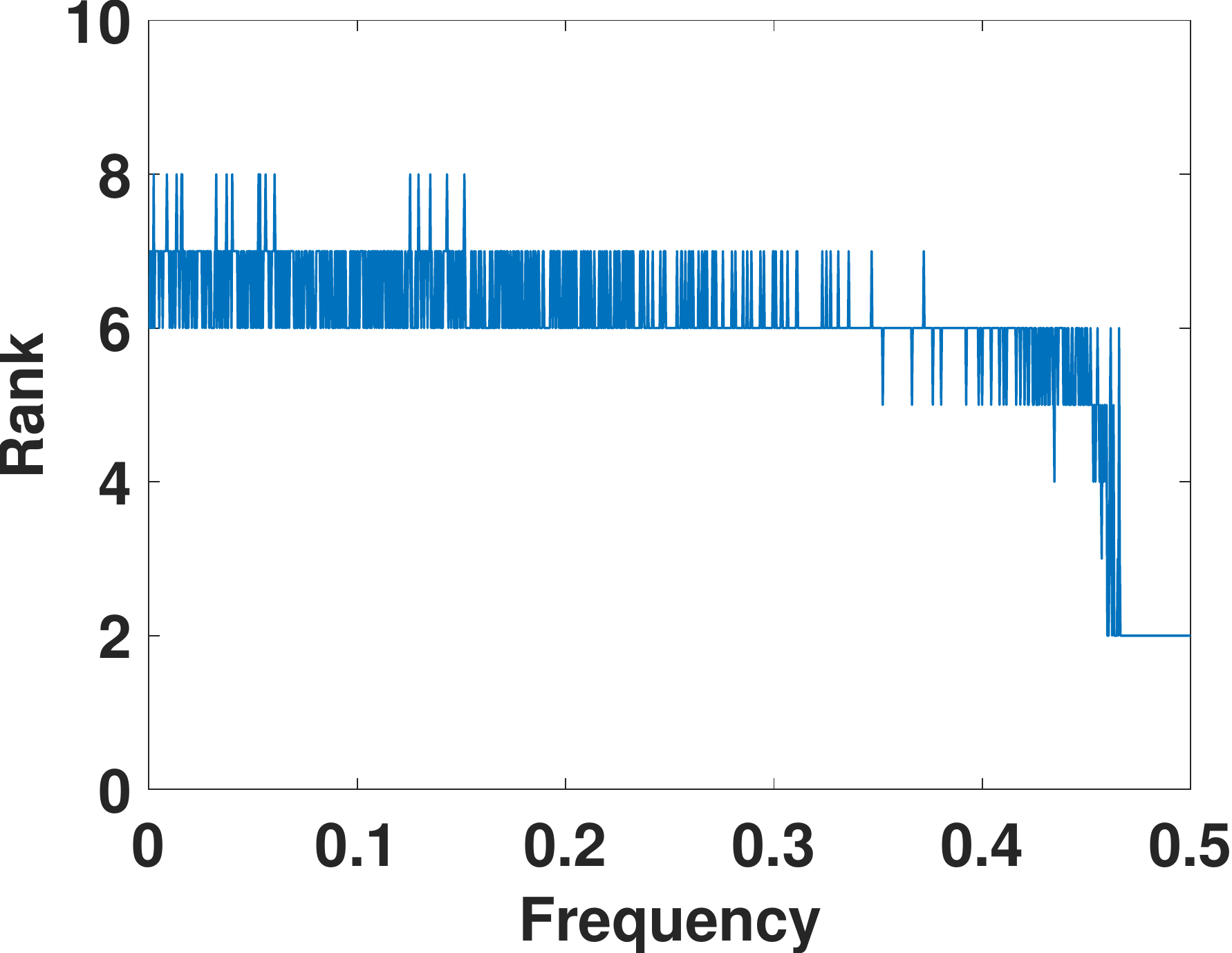}
    \caption{} \label{subfig_2f} \end{subfigure}
    \caption{The repressilator trained from outside a unit ball centered at $0$ and predicted from another basin of attraction using the Koopman operator.  (a) States oscillating with time as simulated (solid lines). Koopman operator trained up to $t = 25$ for prediction (dotted lines) then onward. (b) Simulated (solid) trajectories growing into limit cycles and prediction (dotted) of the limit cycles. (c) Simulated (solid) and predicted (dotted) trajectory form another basin of attraction. (d) Periodogram of trajectories. (e) Rank of Power Spectrum of trajectories up-to $t = 25$ used to train the Koopman Operator}
    \label{fig:RepressilatorPlot_Outside}
\end{figure}
There are many models of the repressilator, with varying degrees of complexity and intricacy to capture the underlying biophysical dynamics. We 
consider a simplified three dimensional model from the first experimental implementation of the repressilator \cite{ElowitzNature}, that captures the limit cycle and basin of attraction, to study the role of the initial condition in PE of the nonlinear system. 
Consider the model:
\begin{equation}
\begin{aligned}
    \dot{m}_{(i)} &= -m_{(i)} + \frac{\alpha}{1 + p_{(j)}^n} + \alpha_0 \\ 
    \dot{p}_{(i)} & = -\beta (p_{(i)} - m_{(i)})\\
    (i,j) &= \{([lacI],[cI]),([tetR],[lacI]),([cI],[tetR])\}\\
    n &= 2, \ \alpha_0 = 0, \ \alpha = 100, \ \beta = 1
    \end{aligned}
\end{equation}

%where $\alpha_x =0.0  nM^{n_x+1}$, $\alpha_y = 0.0  nM^{n_y+1} $, $\alpha_z =0.0  nM^{n_z+1}$, $n_x = 1.0$, $n_y= 1.0$, and $n_z = 1.0.$

An example set of commonly used initial concentrations is $1,0,0,0,0,0$ nM for LacI, $\lambda$-cI, TetR, mLacI, m$\lambda$-cI, and mTetR. We model the degradation and dilution rate of all proteins as a lump term with average kinetic rate $\delta = 0.5$. Fig. \ref{subfig_1a} (and \ref{subfig_2a}) shows simulations (solid lines) of the repressilator from different initial conditions.  The repressilator exhibits a strongly attracting limit cycle and a single unstable equilibrium point at the origin.  Several initial conditions in the phase space mapped through the observable function have low gain, specifically those within $B_1(0)$ (unit ball in $\mathbb{R}^6$). These observable functions have low gain power spectra with spectral lines on the magnitude of numerical noise.  We noted that these initial conditions, when mapped through higher order polynomials, lead to overfitting due to the vanishing of the signal in higher-order terms.   To illustrate, $\bm K$ has been obtained using eDMD with third-order Hermite polynomials. 

%\addtolength{\textheight}{-3cm}   % This command serves to balance the column lengths
                                  % on the last page of the document manually. It shortens
                                  % the textheight of the last page by a suitable amount.
                                  % This command does not take effect until the next page
                                  % so it should come on the page before the last. Make
                                  % sure that you do not shorten the textheight too much.
We considered training the repressilator model with initial conditions drawn from two different regions of the phase space.  First, initial conditions within the unit disc centered at the unstable equilibrium point (solid lines in Fig. \ref{subfig_1b}) and secondly, initial conditions outside the unit disc centered at the same point (solid lines in Fig. \ref{subfig_2b}).  We simulated using 6 initial conditions from one basin of attraction and evaluated test predictions using initial conditions from another basin. For example, we would train within the unit disc (Fig. \ref{subfig_1b}) or outside it (Fig. \ref{subfig_2b}) and evaluate  prediction accuracy of the Koopman operator for trajectories initiated outside that basin of attraction (Fig. \ref{subfig_1c} and \ref{subfig_2c} respectively)  using a norm based error between the predicted and simulated trajectories.  

Notice the rank is greater for the power spectrum in Figure 1 than in Figure 2, which correlates with the failure to predict long-term global behavior in Figure 2.  Interestingly, the rank of the power spectrum was not as high as the dimension of the list of dictionaries, indicating that the true Koopman observable space is of a lower dimension than the dimension of dictionary functions.   Future work will investigate the iterative processes for identifying the minimal set of Koopman observables, their relationship to HankelDMD \cite{HankelDMD}, as well as the formal design of automated biological experiments to ensure {\it global} predictive accuracy of discovered Koopman models.

\section{ACKNOWLEDGMENTS}
The authors would also like to thank Igor Mezic, Robert Egbert, Bassam Bamieh, Sai Pushpak, Sean Warnick, and Umesh Vaidya for stimulating conversations. Any opinions, findings and conclusions or recommendations expressed in this material are those of the author(s) and do not necessarily reflect the views of the Defense Advanced Research Projects Agency (DARPA), the Department of Defense, or the United States Government. This work was supported partially by a Defense Advanced Research Projects Agency (DARPA) Grant No. DEAC0576RL01830 and an Institute of Collaborative Biotechnologies Grant.

%%%%%%%%%%%%%%%%%%%%%%%%%%%%%%%%%%%%%%%%%%%%%%%%%%%%%%%%%%%%%%%%%%%%%%%%%%%%%%%%
\bibliographystyle{ieeetr}
\bibliography{main}
\end{document}